\documentclass[11pt,a4paper,reqno]{amsart}

\usepackage[a4paper,lmargin=2.5cm,rmargin=2cm,tmargin=4cm,bmargin=4cm]{geometry}

\makeatletter
\g@addto@macro\bfseries{\boldmath} 
\makeatother

\usepackage[centertags]{amsmath}
\usepackage{amsfonts}
\usepackage{amssymb}
\usepackage{amsthm}

\usepackage{hyperref}
	\hypersetup{breaklinks=true,colorlinks=true,
linkcolor=MidnightBlue,citecolor=MidnightBlue,
urlcolor=MidnightBlue}

\usepackage{tikz}
\usepackage{tikz-cd}
\usepackage[normalem]{ulem}
\usepackage[shortlabels]{enumitem}

\usepackage[dvipsnames]{xcolor}

\usepackage{mathtools,dsfont}

\usepackage{mathrsfs}
\usepackage{graphicx}

\usepackage{orcidlink}

\usepackage{todonotes}

\newtheorem{theorem}{Theorem}[section]

\newtheorem{lemma}[theorem]{Lemma}

\theoremstyle{definition}
\newtheorem{definition}[theorem]{Definition}

\theoremstyle{remark}
\newtheorem{remark}[theorem]{Remark}

\newtheorem{fact}[theorem]{Fact}

\numberwithin{equation}{section}

\newcommand{\Z}{\mathbb{Z}}
\newcommand{\N}{\mathbb{N}}
\newcommand{\K}{\mathbb{K}}

\newcommand{\T}{\mathbb{T}}

\newcommand{\vertiii}[1]{{\left\vert\kern-0.25ex\left\vert\kern-0.25ex\left\vert #1 
		\right\vert\kern-0.25ex\right\vert\kern-0.25ex\right\vert}}

\makeatletter
\newcommand{\markthis}[3]{
	\overset{
		\textup{\makebox[0pt]{#1}}%
		\def\@currentlabel{#1}%
		\ltx@label{#2}%
	}{
		#3%
	}%
}

\DeclareMathOperator{\re}{Re}

\newcommand{\nn}[1]{{\left\vert\kern-0.25ex\left\vert\kern-0.25ex\left\vert #1 
		\right\vert\kern-0.25ex\right\vert\kern-0.25ex\right\vert}}

\newcommand{\nnn}[1]{{\left\vert\kern-0.25ex\left\vert\kern-0.25ex\left\vert #1 
		\right\vert\kern-0.25ex\right\vert\kern-0.25ex\right\vert_1}}

\newcommand{\nnnn}[1]{{\left\vert\kern-0.25ex\left\vert\kern-0.25ex\left\vert #1 
		\right\vert\kern-0.25ex\right\vert\kern-0.25ex\right\vert_2}}

        \newcommand{\abs}[1]{\left\vert#1\right\vert}
        
\renewcommand{\geq}{\geqslant}
\renewcommand{\leq}{\leqslant}

\newcommand{\norm}[1]{\left\Vert#1\right\Vert}

\newcommand{\spann}{\operatorname{span}}

\newcommand{\supp}{\operatorname{supp}}

\newcommand{\eps}{\varepsilon}

\newcounter{smallromans}

	{\end{list}}

\makeatletter
\renewcommand{\tocsection}[3]{%
	\indentlabel{\@ifnotempty{#2}{\bfseries\ignorespaces#1 #2\quad}}\bfseries#3}
\renewcommand{\tocsubsection}[3]{%
	\indentlabel{\@ifnotempty{#2}{\ignorespaces#1 #2\quad}}#3}

\renewcommand\@dotsep{4.5}
\def\@tocline#1#2#3#4#5#6#7{\relax
	\ifnum #1>\c@tocdepth 
	\else
	\par \addpenalty\@secpenalty\addvspace{#2}%
	\begingroup \hyphenpenalty\@M
	\@ifempty{#4}{%
		\@tempdima\csname r@tocindent\number#1\endcsname\relax
	}{%
		\@tempdima#4\relax
	}%
	\parindent\z@ \leftskip#3\relax \advance\leftskip\@tempdima\relax
	\rightskip\@pnumwidth plus1em \parfillskip-\@pnumwidth
	#5\leavevmode\hskip-\@tempdima{#6}\nobreak
	\leaders\hbox{$\m@th\mkern \@dotsep mu\hbox{.}\mkern \@dotsep mu$}\hfill
	\nobreak
	\hbox to\@pnumwidth{\@tocpagenum{\ifnum#1=1\bfseries\fi#7}}\par
	\nobreak
	\endgroup
	\fi}
\AtBeginDocument{%
	\expandafter\renewcommand\csname r@tocindent0\endcsname{0pt}
}
\def\l@subsection{\@tocline{2}{0pt}{2.5pc}{5pc}{}}
\makeatother

\begin{document}

	\title[A strictly convex and smooth Banach space with the Daugavet property]{A strictly convex and smooth Banach space with the Daugavet property}

\date{\today}

\author[Dantas]{Sheldon Dantas}
\address[Dantas]{Czech Technical University in Prague, FEE, Department of Mathematics, Technick\'a 2, 16627, Prague 6, Czech Republic. \newline
\href{https://orcid.org/0000-0001-8117-3760}{ORCID: \texttt{0000-0001-8117-3760}}}
\email{\texttt{sheldon.dantas@fel.cvut.cz}}
\urladdr{www.sheldondantas.com}

\author[Kaasik]{Jaan Kristjan Kaasik}
\address[Kaasik]{Institute of Mathematics and Statistics. University of Tartu, Narva Mnt 18, 51009 Tartu, Estonia. \newline
\href{https://orcid.org/0000-0001-6561-5557}{ORCID: \texttt{0000-0001-6561-5557}}}
\email{\texttt{jaan.kristjan.kaasik@ut.ee}}

\author[Perreau]{Yo\"el Perreau}
\address[Perreau]{Institute of Mathematics and Statistics, University of Tartu, Narva mnt 18, 51009 Tartu, Estonia}
\email{yoel.perreau@ut.ee}
\urladdr{
\href{https://orcid.org/0000-0002-2609-5509}{ORCID: \texttt{0000-0002-2609-5509} } }

\begin{abstract} We construct equivalent norms on $L_1[0,1]$, arbitrarily close to the usual $L_1$-norm, which are weakly midpoint locally uniformly rotund and Gâteaux smooth, and whose dual norms have the Daugavet property. In particular, we obtain a strictly convex and smooth Banach space with the Daugavet property, thereby solving a longstanding open problem in the area.
\end{abstract}

	
	\maketitle

\tableofcontents

\section{Introduction}

The Daugavet property (DPr, for short) is a remarkable geometric property of Banach spaces which originates from the classical norm identity
\begin{equation*}
    \|I+T\|=1+\|T\|
\end{equation*}
for compact operators $T$ on the space $C[0,1]$, which was first observed by I.K.~Daugavet in \cite{Daugavet63}. This phenomenon was later generalized to a rich family of classical Banach spaces such as $C(K)$-spaces for perfect compact Hausdorff spaces $K$ and $L_1(\mu)$-spaces for atomless measures $\mu$ \cite{KSSW00}, and Lipschitz free-spaces $\mathcal{F}(M)$ for length metric spaces $M$ \cite{GPRZ18}, as well as some more exotic examples, such as the Bourgain--Rosenthal space \cite{KW04} and a quotient of $L_1[0,1]$ obtained in \cite[Theorem~3.4]{KSSW00} using Talagrand's construction \cite{Tal90}. The DPr has also been studied in non-commutative settings, including $C^*$-algebras, non-commutative $L_p$-spaces, preduals of von Neumann algebras, and $JB^*$-triples and their isometric preduals \cite{Oik02,BGM05,MU14}. More recently, the DPr has also been investigated in symmetric projective tensor products \cite{MRZ22} and vector-valued Lipschitz spaces \cite{MRZ25}. It has therefore generated an intense research activity among people working in various areas of functional analysis over the last three decades. We refer to the recent monograph \cite{KMRZW25} for a thorough coverage of the history as well as of the most recent advances on the topic.

Since the discovery of the celebrated geometric characterization of the DPr from \cite{KSSW00}, it has become clear that the DPr imposes very strong constraints on the isomorphic structure and the geometry of the unit ball of the underlying Banach space \cite{KSSW00,Shv00,LRZ2021,KMRZW25}. In particular, spaces with the Daugavet property present a strong diameter two behavior and hence fail the Radon-Nikodým property \cite{KSSW00,KMRZW25}; contain isomorphic and, actually, asymptotically isometric copies of the space $\ell_1$ \cite[Theorem~2.9]{KSSW00}; and do not embed isomorphically in Banach spaces with an unconditional basis \cite{KSSW00,Shv00}. These properties seem hardly compatible with any properties of rotundity or smoothness of the norm, and the question of the existence of a strictly convex Banach space with the Daugavet property was in fact already raised thirty years ago by V.M.~Kadets in \cite{Kad96} (see also \cite[Questions~3.3 and~3.4]{KMRZW25}). Let us point out that the dual of a Banach space with the DPr is neither smooth nor strictly convex (see \cite[Theorem 2.1]{KMMP09}).

In the absence of completeness, the strictly convex case was already known to have a positive answer \cite{Kad96,KMMP09}, and, more recently, a non-complete smooth normed space with the Daugavet property has also been constructed \cite{Ham26-1}. Furthermore, geometric constructions of non-complete smooth and strictly convex normed spaces with the Daugavet property can be obtained in a rather straightforward way (unpublished work) by using inductively the fundamental property of strong Gurarii spaces together with a finite-dimensional approach of the Daugavet property in the spirit of \cite{Kad96} or \cite{KW04}. In all three cases, passing to the completion seems to ruin strict convexity and/or smoothness of the norm, and the Banach space case has remained open. In fact, it was not even known if there could exist a strictly convex or smooth Banach space with a Daugavet point in the sense of \cite{AHLP20} (see \cite[Question~7.5]{MPRZ24} and \cite{ALMP22,AALMPPV24,HLPV23}). On the other hand, it should be pointed out that weaker forms of diameter two properties were known for a long time to be compatible with some rather strong properties of rotundity of the norm for Banach spaces. For instance, the quotient $C(\T)/A$, where $C(\T)$ is the space of continuous functions on the complex unit circle $\T$ and $A$ is the disc algebra, is known to be strictly convex (in fact wMLUR) and to satisfy the diameter two property \cite{AHNTT16} (see also \cite{NPTV24,CH25,DBPS25,AHLT26} for more recent examples and advances in this direction). Nevertheless, a quick study of the aforementioned constructions yields that they fail to have the DPr and, in most cases, they fail it in the stronger sense that their unit sphere contains no Daugavet points.

The purpose of the present paper is to answer both questions affirmatively and, in fact, simultaneously. Given a probability space $(\Omega,\Sigma,\mu)$ and $\eps>0$, we construct an equivalent norm $\nn{\cdot}$ on $L_1(\mu)$ which is $(1+\eps)$ equivalent to the original $L_1$-norm  by means of an invertible measure-preserving transformation $T:\Omega\rightarrow\Omega$. Under the extra assumption that $T$ is ergodic, we prove that the resulting norm is wMLUR and Gâteaux smooth. In particular, it is strictly convex. On the other hand, and without assuming ergodicity, we prove that the space $X:=(L_1(\mu),\nn{\cdot})$, and also its dual space $X^*=(L_{\infty}(\mu),\nn{\cdot}_{*})$, has the Daugavet property whenever $\mu$ is atomless. Therefore, an ergodic transformation produces a Banach space which is simultaneously wMLUR, Gâteaux smooth and has the Daugavet property.

The proof of the Daugavet property is based on decompositions into vectors with arbitrarily small supports together with suitable estimates, which is the standard way to prove that Lipschitz free-spaces or $L_1(\mu)$-spaces satisfy the DPr. We also show that the space $X$ satisfies the stronger operator Daugavet property (ODPr, for short). Finally, we observe that the bidual $X^{**}$ fails to have the Daugavet property, and that $X$ fails to be MLUR.

\section{Background}

Throughout the paper, all the Banach spaces will be considered either in the real or complex field. Let $X$ be a Banach space. We denote respectively by $B_X$ and $S_X$ the closed unit ball and the unit sphere of $X$. The symbol $X^*$ stands for the topological dual of $X$ and $\mathcal{L}(X)$ for the Banach space of all bounded linear operators from $X$ into $X$. If $x^*\in S_{X^*}$ and $\alpha>0$, we denote by \begin{equation*}
    S(B_X,x^*,\alpha) := \{x\in B_X:\re x^*(x) > 1 - \alpha \}
\end{equation*} the corresponding slice of $B_X$. For any sets $A$ and $B$, we set 
\begin{equation*}
    A \Delta B := (A \setminus B) \cup (B \setminus A).
\end{equation*}

The main definitions of the paper are the following ones. We refer to the monographs \cite{DGZ93,KMRZW25} for background. Weakly midpoint locally uniformly rotund spaces will be shortened to wMLUR.

\begin{definition} Let $X$ be a Banach space. We say that $X$ is
\begin{itemize}

\item[(a)] {\bf wMLUR} if, whenever $x \in S_X$ and $(x_n) \subseteq S_X$ satisfy $\|\frac{x + x_n}{2}\| \rightarrow 1$, then $x_n \stackrel{w}{\longrightarrow} x$;
\item[(b)] {\bf strictly convex} if, for every $x, y \in S_X$ with $x \neq y$, we have that $\|\frac{x+y}{2}\| < 1$;

\item[(c)] {\bf Gâteaux smooth} if, for every $x \in S_X$ and every $y \in X$, the following limit exists
\begin{equation*}
    \lim_{t\to 0}\frac{\|x+ty\|-\|x\|}{t}.
\end{equation*}
\end{itemize}
\end{definition}
It is immediate that every wMLUR norm is strictly convex. A Banach space $X$ is said to have the {\bf Daugavet property} (DPr, for short) if
\begin{equation*}
    \|I+T\|=1+\|T\|
\end{equation*}
for every rank-one operator $T\in\mathcal{L}(X)$. It is known that the space $X$ has the DPr if and only if, for every $x\in S_X$, every slice
$S\subset B_X$, and every $\eps>0$, there exists $y\in S$ such that $\|x+y\|>2-\eps$. It is clear that if $X^*$ has the DPr, then so does $X$, but the converse is false in general, as e.g. the dual of $C[0,1]$ is known to fail the DPr.

\section{Main results}
\subsection{Definition of the norm} Let $(\Omega, \Sigma, \mu)$ be a probability space and $T: \Omega \rightarrow \Omega$ be an {\bf invertible measure-preserving transformation}, that is a one-to-one and onto mapping $T:\Omega\to \Omega$ such that $T$ and $T^{-1}$ are measurable and $\mu(T^{-1}(A)) = \mu(A)$ for every $A \in \Sigma$. Throughout the rest of the text, $T: \Omega \rightarrow \Omega$ will denote an invertible measure-preserving transformation.  Then, let $(a_k)_{k \in \Z}$ be a sequence of positive real numbers such that 
\begin{equation*}
    \sum_{k \in \Z} a_k = 1 \ \ \ \mbox{and} \ \ \ \sum_{k \in \Z} \sqrt{a_k} < \infty.
\end{equation*}
For $f \in L_1(\mu)$, we consider the vector-valued function $Jf: \Omega \rightarrow \ell_2(\Z)$ given by
\begin{equation*}
    Jf(\omega) := ( \sqrt{a_k} f(T^k \omega) )_{k \in \Z}.
\end{equation*}
Then, we define
\begin{equation*}
    \nn{f} := \int_{\Omega} \|Jf(\omega)\|_2 d \mu(\omega) = \int_{\Omega} \left( \sum_{k \in \Z} a_k |f(T^k \omega)|^2 \right)^{1/2} d \mu (\omega).
\end{equation*}
It is straightforward to check that $\nn{\cdot}$ defines an equivalent norm on $L_1(\mu)$ with 
\begin{equation} \label{equivalence}
    \|f\|_1 \leq \nn{f} \leq \left( \sum_{k \in \Z} \sqrt{a_k} \right) \|f\|_1 
\end{equation}
for every $f \in L_1(\mu)$. Note that by a suitable choice of $(a_k)_{k\in\Z}$, we can make this norm arbitrarily close to the original $L_1$-norm. Also, let us note that 
\begin{equation*}
    \nn{f \chi_E} \leq \left( \sum_{k \in \Z} \sqrt{a_k} \right) \int_E |f| d \mu 
\end{equation*}
for every $f\in L_1(\mu)$ and $E\in\Sigma$, and hence $\nn{f \chi_E} \rightarrow 0$ whenever $\mu(E) \rightarrow 0$. We will also need the following fact.

\begin{fact} \label{erg-monotone} If $f, g \in L_1(\mu)$ satisfy $|f| \leq |g|$ almost everywhere on $\Omega$, then $\nn{f} \leq \nn{g}$. Moreover, if $\mu(\{ \omega: |f(\omega)| < |g(\omega)| \}) > 0$, then $\nn{f} < \nn{g}$.
\end{fact}

\begin{proof} By assumption, there exists a null set $N \in\Sigma$ such that $|f(\omega)| \leq |g(\omega)|$ for every $\omega \in \Omega \setminus N$. Since $T$ is measure-preserving, for every $k \in \Z$, we have that $\mu(T^{-k} N) = \mu(N) = 0$. Hence, $\mu(\bigcup_{k \in \Z} T^{-k} N ) = 0$. Therefore, for almost every $\omega$, $T^k \omega \not\in N$ for every $k \in \Z$, that is, $|f(T^k \omega)| \leq |g(T^k \omega)|$ for every $k \in \Z$. Since $a_k > 0$ for every $k \in \Z$, integrating over $\Omega$ gives us the first part of the result. For the strict part, since $a_0 > 0$, for every $\omega \in \mu(\{ \omega: |f(\omega)| < |g(\omega)| \})$ outside the null set we have $a_0 |f(\omega)|^2 < a_0|g(\omega)|^2$ while $a_k |f(T^k \omega)|^2 \leq a_k |g(T^k \omega)|^2$ for every $k \in \N$. This means that $\|Jf(\omega)\|_2 < \|Jg(\omega)\|_2$ on a set of positive measure. Therefore, $\nn{f} < \nn{g}$.
\end{proof} 

Throughout the rest of this paper we denote $X$ as the Banach space $X:= (L_1(\mu), \nn{\cdot})$. Since $\mu$ is a probability measure, we can identify $X^*$ with $L_{\infty}(\mu)$ as a vector space. For $f \in L_1(\mu)$ and $\Phi \in L_{\infty}(\mu)$, we write 
\begin{equation*}
    \Phi(f) = \int_{\Omega}  \Phi f d \mu \ \ \ \mbox{and} \ \ \ \nn{\Phi}_{*} = \sup_{\nn{f} \leq 1} | \Phi(f)|.
\end{equation*}
Clearly, the dual norm satisfies 
\begin{equation*}
    \frac{1}{\sum_{k \in \Z} \sqrt{a_k}} \|\Phi\|_{\infty} \leq \nn{\Phi}_{*} \leq \|\Phi\|_{\infty} 
\end{equation*}
for every $\Phi \in L_{\infty}(\mu)$. Throughout the manuscript, we will use all these facts without any explicit reference.

\subsection{Strict convexity and smoothness are equivalent to ergodicity} In this subsection, we will prove that $(L_1(\mu), \nn{\cdot})$ is both strictly convex and Gâteaux smooth whenever the transformation $T: \Omega \rightarrow \Omega$ satisfies some additional ergodicity assumption. We have the following definition.

\begin{definition} The transformation $T$ is {\bf ergodic} if every $A \in \Sigma$ satisfying $\mu(A\Delta T^{-1}(A)) = 0$ also satisfies $\mu(A) = 0$ or $\mu(A) = 1$.
\end{definition}

We do not know if every probability measure space admits an ergodic transformation, but such mappings are known to exist in the case $\mu$ is the Lebesgue measure on $[0,1]$. We will use the following fact throughout the text. 

\begin{fact}
    Let $T:\Omega\to\Omega$ be an ergodic transformation and $A\in\Sigma$ with $\mu(A)>0$. Then, the set \begin{equation*}
        \bigcup_{k\in\Z}T^{-k}(A)
    \end{equation*} has full measure.
\end{fact}

\begin{proof}
    Let $B:=\bigcup_{k\in\Z}T^{-k}(A)$. By definition, we have $T^{-1}(B)=B$. Furthermore, since $A\subset B$, then $\mu(B)>0$, so ergodicity of $T$ implies that $\mu(B)=1$. 
\end{proof}

In particular, let us note than in case $T$ is ergodic, we have that for every non-zero $f\in L_1(\mu)$, 
\begin{equation*}
    \bigcup_{k \in \Z} \{ \omega\in\Omega: f(T^k\omega) \not= 0 \}  
\end{equation*}
has full measure. Consequently, $Jf(\omega) \not=0$ almost everywhere. 

We have the following characterization.

\begin{theorem} \label{theorem-SC-Gateaux} Let $(\Omega, \Sigma, \mu)$ be a probability space and $T: \Omega \rightarrow \Omega$ be an invertible measure-preserving transformation. The following assertions are equivalent.

\begin{itemize}
\item[(1)] $T$ is ergodic.
\item[(2)] $\nn{\cdot}$ is strictly convex.
\item[(3)] $\nn{\cdot}$ is Gâteaux smooth.
\item[(4)] $X$ does not contain an isometric copy of $\ell_1^2$.
\end{itemize}
\end{theorem}

\begin{proof} First, let us consider the case where $T$ is not ergodic. Then there exists $E\in\Sigma$ such that $0<\mu(E)<1$ and $\mu(E\Delta T^{-1}(E))=0$. Since $T$ is measure preserving, the latter simply means that $T\omega \in E$ for almost every $\omega\in E$ and $T\omega \in \Omega\setminus E$ for almost every $\omega \in \Omega\setminus E$. Iterating, we get that for almost every $\omega\in E$ , $T^k\omega \in E$ for every $k\in\Z$, and that for almost every $\omega \in \Omega\setminus E$, $T^k\omega \in \Omega\setminus E$ for every $k\in\Z$. Let $f:=\chi_E$ and $g:=\chi_{\Omega\setminus E}$. Since $0<\mu(E)<1$, we have that the functions $f$ and $g$ are non-zero vectors in $L_1(\mu)$. Furthermore, by the above, we have that the function $Jf$ is supported on $E$ while the function $Jg$ is supported on $\Omega\setminus E$. So for every $\alpha,\beta\in \K$, we have \begin{equation*}
    \nn{\alpha f+\beta g}=\abs{\alpha}\nn{f}+\abs{\beta}\nn{g},
\end{equation*} which means that the space $Y:=\spann\{f,g\}$ is isometrically isomorphic to the space $\ell_1^2$. In particular, $\nn{\cdot}$ is neither smooth nor strictly convex. 

It remains to prove that in the case where $T$ is ergodic, the norm $\nn{\cdot}$ is simultaneously strictly convex and Gâteaux smooth. We start with strict convexity. Let $f, g \in L_1(\mu)$ be such that $f\not= g$ and $\nn{f} = \nn{g} = 1$. We need to show that $\nn{f+g} < 2$. There are three main cases to consider. Assume first that $\mu(\supp f \setminus \supp g) > 0$. Since $\nn{g} = 1$, we have that $g \not= 0$ and $\mu(\supp g)  > 0$. Since $T$ is ergodic, the sets 
\begin{equation*}
    \bigcup_{k \in \Z} T^{-k}(\supp f \setminus \supp g) \ \ \ \mbox{and} \ \ \ \bigcup_{k \in \Z} T^{-k}(\supp g)
\end{equation*}
have full measure. This means that, for almost every $\omega$, one coordinate of $Jf(\omega)$ is nonzero while the corresponding coordinate of $Jg(\omega)$ is zero and also $Jg(\omega)$ has a nonzero coordinate. In particular, the vectors $Jf(\omega)$ and $Jg(\omega)$ are linearly independent, and strict convexity of $\ell_2(\Z)$ gives 
\begin{equation*}
    \|Jf(\omega) + Jg(\omega)\|_2 < \|Jf(\omega)\|_2 + \|Jg(\omega)\|_2
\end{equation*}
almost everywhere. Integration then yields $\nn{f+g} < 2$. The exact same argument applies when $\mu(\supp g \setminus \supp f) > 0$, so it only remains to consider now the case when $\supp f = \supp g$, up to a null set. Suppose that $\mu(\{ \omega: |f(\omega)| < |g(\omega)|\}) > 0$. Then the set 
\begin{equation*}
    \supp g \setminus \{ \omega\in\Omega: |f(\omega)| < |g(\omega)|\} 
\end{equation*}
also has positive measure. Indeed, otherwise, $|f| < |g|$ almost everywhere on $\supp g$ and outside $\supp g$, $f$ and $g$ both vanish since they have the same support. So $0 \leq |f| \leq |g|$ almost everywhere. By Fact \ref{erg-monotone}, this would imply that $1 = \nn{f} < \nn{g} = 1$, a contradiction. Ergodicity gives 
\begin{equation*}
    \mu \left( \bigcup_{k \in \Z} T^{-k} ( \{ \omega\in\Omega: |f(\omega)| < |g(\omega)| \}) \right) = 1 = \mu \left( \bigcup_{l \in \Z} T^{-l} ( \supp g \setminus \{ \omega\in\Omega: |f(\omega)| < |g(\omega)| \} ) \right).
\end{equation*}
Their intersection also has full measure. Therefore, for almost every $\omega$, there are $k, l \in \Z$ such that 
\begin{equation*}
|f(T^k\omega)| < |g(T^k\omega)| \ \ \ \mbox{and} \ \ \ |f(T^l\omega)| \geq |g(T^l\omega)| > 0.    
\end{equation*}
Thus, the $k$-th coordinates of $Jf(\omega)$ and $Jg(\omega)$ satisfy
\begin{equation*}
    \sqrt{a_k} |f(T^{k} \omega)| < \sqrt{a_k} |g(T^k \omega)| 
\end{equation*}
while the $l$-th coordinates satisfy
\begin{equation*}
    \sqrt{a_l} |f(T^l \omega)| \geq \sqrt{a_l} |g(T^l \omega)| > 0.
\end{equation*}
Therefore, $Jf(\omega)$ and $Jg(\omega)$ are linearly independent. So, strict convexity of $\ell_2(\Z)$ gives again 
\begin{equation*}
    \|Jf(\omega) + Jg(\omega)\|_2 < \|Jf(\omega)\|_2 + \|Jg(\omega)\|_2
\end{equation*}
almost everywhere. Integration gives $\nn{f + g} < 2$.

It remains to consider the case $\mu (\{ \omega\in\Omega: |f(\omega)| < |g(\omega)| \}) = 0$. Then, $|g| \leq |f|$ almost everywhere. On the other hand, $\nn{|f|} = \nn{f} = \nn{g} = \nn{|g|}$ and the strict part of Fact \ref{erg-monotone} give $|f| = |g|$ almost everywhere. Since $f\not=g$, we have that 
\begin{equation*}
    \mu( \{ \omega\in\Omega: f(\omega) \not= g(\omega) \} ) > 0.
\end{equation*}
Once again, ergodicity yields 
\begin{equation*}
    \mu \left( \bigcup_{k \in \Z} T^{-k} (\{ \omega\in\Omega: f(\omega) \not= g(\omega) \}) \right) = 1.
\end{equation*}
Also, since $|f| = |g|$ almost everywhere, $|f(T^k \omega)| = |g(T^k \omega)|$ for every $k \in \Z$ almost everywhere. So, $\|Jf(\omega)\|_2 = \|Jg(\omega)\|_2>0$ almost everywhere and there exists $k \in \Z$ such that the $k$-th coordinates of $Jf(\omega)$ and $Jg(\omega)$ are different. So, $Jf(\omega) \not= Jg(\omega)$, and we get once again 
\begin{equation*}
    \|Jf(\omega) + Jg(\omega)\|_2 < \|Jf(\omega)\|_2 + \|Jg(\omega)\|_2
\end{equation*}
almost everywhere. Integration again gives $\nn{f+g} < 2$, and strict convexity of $\nn{\cdot}$ follows.

 Finally, let us deal with the smoothness of $\nn{\cdot}$. Let $f\not= 0$. As previously noticed, ergodicity gives that
\begin{equation*}
    \bigcup_{k \in \Z} T^{-k} (\{ \omega\in\Omega: f(\omega) \not= 0 \} ) 
\end{equation*}
has full measure. Consequently, $Jf(\omega) \not=0$ almost everywhere. Since $J$ is linear, we have that 
\begin{equation*}
    J(f + th)(\omega) = Jf(\omega) + t Jh(\omega)
\end{equation*}
for every $h \in L_1(\mu)$ and every $\omega$. Thus, for $t \not=0$,
\begin{equation*}
    \frac{\nn{f+ th} - \nn{f}}{t} = \int_{\Omega} \frac{\|Jf(\omega) + t Jh(\omega)\|_2 - \|Jf(\omega)\|_2}{t}d\mu(\omega).
\end{equation*}
Now observe after multiplying the numerator by $\|Jf(\omega) + t Jh(\omega)\|_2 + \|Jf(\omega)\|_2$, we obtain
\begin{equation*}
    \frac{\|Jf(\omega) + t Jh(\omega)\|_2 - \|Jf(\omega)\|_2}{t} = \frac{2 \re \langle Jh(\omega), Jf(\omega)\rangle + t \|Jh(\omega)\|_2^2}{\|Jf(\omega) + tJh(\omega)\|_2 + \|Jf(\omega)\|_2},
\end{equation*}
and therefore 
\begin{equation*}
    \lim_{t \rightarrow 0} \frac{\|Jf(\omega) + t Jh(\omega)\|_2 - \|Jf(\omega)\|_2}{t} = \re \frac{\langle Jh(\omega), Jf(\omega)\rangle}{\|Jf(\omega)\|_2}.
\end{equation*}
By the reverse triangle inequality, we have that 
\begin{equation*}
    \left| \frac{\|Jf(\omega) + t Jh(\omega)\|_2 - \|Jf(\omega)\|_2}{t} \right| \leq \|Jh(\omega)\|_2.
\end{equation*}
So by the Dominated Convergence Theorem, we can pass the limit through the integral, and obtain \begin{equation*}
     \lim_{t \rightarrow 0} \frac{\nn{f + th} - \nn{f}}{t} = \int_\Omega \re \frac{\langle Jh(\omega), Jf(\omega)\rangle}{\|Jf(\omega)\|_2} d\mu(\omega).
\end{equation*} 
Thus, $\nn{\cdot}$ is Gâteaux smooth.
\end{proof}

It turns out that when $T$ is ergodic, we get a little more than strict convexity of the norm.

\begin{theorem}
     Let $(\Omega, \Sigma, \mu)$ be a probability space and $T: \Omega \rightarrow \Omega$ be an ergodic invertible measure-preserving transformation. Then $\nn{\cdot}$ is wMLUR.
\end{theorem}

\begin{proof}
    Let $f \in S_X$ be given and suppose that $(h_n) \subseteq L_1(\mu)$ satisfies $\nn{f + h_n} \rightarrow 1$ and $\nn{f - h_n} \rightarrow 1$ as $n \rightarrow \infty$. For every $\omega \in \Omega$, we can write 
\begin{equation*}
    2 Jf(\omega) = (Jf(\omega) + J h_n(\omega)) + (Jf(\omega) - Jh_n(\omega))
\end{equation*}
which implies that 
\begin{equation} \label{wLUReq1}
    2\|Jf(\omega)\|_2 \leq \|Jf(\omega) + Jh_n(\omega)\|_2 + \|Jf(\omega) - Jh_n(\omega)\|_2. 
\end{equation}
Similarly, for every $\omega \in \Omega$, we have that 
\begin{equation*}
      2\|Jh_n(\omega)\|_2 \leq \|Jf(\omega) + Jh_n(\omega)\|_2 + \|Jf(\omega) - Jh_n(\omega)\|_2.   
\end{equation*}
This implies that 
\begin{equation} \label{wLUReq2}
    \|Jh_n(\omega)\|_2 \leq \|Jf(\omega)\|_2 + \frac{1}{2} \Big( \|Jf(\omega) + Jh_n(\omega)\|_2 + \|Jf(\omega) - Jh_n(\omega)\|_2 - 2 \|Jf(\omega)\|_2 \Big).
\end{equation}
The expression inside parentheses in (\ref{wLUReq2}) is non-negative (because of (\ref{wLUReq1})) and then 
\begin{equation}
\begin{aligned}
&\int_{\Omega}\|Jf(\omega)+Jh_n(\omega)\|_2\,d\mu(\omega)\\
&\quad+\int_{\Omega}\|Jf(\omega)-Jh_n(\omega)\|_2\,d\mu(\omega)\\
&\quad-2\int_{\Omega}\|Jf(\omega)\|_2\,d\mu(\omega) =\nn{f+h_n}+\nn{f-h_n}-2\nn{f}
\longrightarrow 0.
\end{aligned} \label{wLUReq3}
\end{equation}
as $n \rightarrow \infty$. The coordinate $Jh_n(\omega)$ indexed by zero is $\sqrt{a_0} h_n(\omega)$ and so $\sqrt{a_0}|h_n(\omega)| \leq \|Jh_n(\omega)\|_2$. By (\ref{wLUReq2}), for every measurable set $A \in \Sigma$, we get 
\begin{equation}
\begin{aligned}
\sqrt{a_0}\int_A |h_n|\,d\mu
&\leq \int_A \|Jh_n(\omega)\|_2\,d\mu(\omega)\\
&\leq \int_A \|Jf(\omega)\|_2\,d\mu(\omega)\\
&\quad +\frac12\int_{\Omega}
    \|Jf(\omega)+Jh_n(\omega)\|_2\,d\mu(\omega)\\
&\quad +\frac12\int_{\Omega}
    \|Jf(\omega)-Jh_n(\omega)\|_2\,d\mu(\omega)\\
&\quad -\int_{\Omega}\|Jf(\omega)\|_2\,d\mu(\omega).
\end{aligned} \label{wLUReq4}
\end{equation}
Since $\omega \mapsto \|Jf(\omega)\|_2$ belongs to $L_1(\mu)$, its integral is absolutely continuous. By (\ref{wLUReq3}), the sum of the last three terms in (\ref{wLUReq4}) goes to zero. This means that $(h_n)$ is uniformly integrable. Now, take $A = \Omega$ in (\ref{wLUReq4}) and we obtain 
\begin{equation*}
    \sqrt{a_0} \|h_n\|_1 \leq \nn{f} + \frac{1}{2} ( \nn{f+h_n} + \nn{f-h_n} - 2\nn{f})
\end{equation*}
and so $(h_n)$ is bounded in $L_1(\mu)$. By \cite[Theorem IV.8.9]{DunfordSchwartz}, $(h_n)$ is relatively $w$-compact in $L_1(\mu)$. Suppose then that a subsequence $(h_{n_j})$ converges weakly to $h$. Then,
\begin{equation*}
    \nn{f + h} \leq \liminf_j \nn{f + h_{n_j}} = 1 \ \ \ \mbox{and} \ \ \ 
    \nn{f - h} \leq \liminf_j \nn{f - h_{n_j}} = 1.
\end{equation*}
On the other hand, 
\begin{equation*}
    2 = \nn{2 f} \leq \nn{f + h} + \nn{f - h} \leq 2
\end{equation*}
which implies that $\nn{f + h} = \nn{f - h} = 1$. Strict convexity gives $f+h=f-h$, that is, $h = 0$. This means that every $w$-cluster point of $(h_n)$ is zero. Since $(h_n)$ is relatively weakly compact, Eberlein--Šmulian gives that $h_n \stackrel{w}{\longrightarrow} 0$.
\end{proof}

\subsection{The space satisfies the Daugavet property} 

Our goal in this subsection is to give a direct proof of the fact that the space $X$ satisfies the Daugavet property provided that the probability measure $\mu$ is atomless. Although the following result can be seen as a direct corollary of Theorems~\ref{theorem:dual-has-the-DPr} or \ref{theorem:X_has_ODP} that we will prove in the following two subsections, we chose to give an independent proof of it as a warm-up for the reader.  Notice that for the Daugavet property we do not need the extra assumption that the transformation $T: \Omega \rightarrow \Omega$ is ergodic.

\begin{theorem} \label{theorem:space-has-the-DPr} Let $(\Omega, \Sigma, \mu)$ be an atomless probability space and $T: \Omega \rightarrow \Omega$ be an invertible measure-preserving transformation. Then the space $X = (L_1(\mu), \nn{\cdot})$ has the DPr.
\end{theorem}

In order to prove Theorem \ref{theorem:dual-has-the-DPr}, we need several lemmas. From now on, we always assume that $\mu$ is an atomless probability measure. Let us start by introducing some useful notation. For $m \in \N$ and $f\in L_1(\mu)$, let us set 
\begin{equation*}
    J_m f(\omega) := (\sqrt{a_k} f (T^k \omega))_{|k| \leq m}
\end{equation*}
for every $\omega \in \Omega$. Also, let us set \begin{equation*}
    \nn{f}_m:= \int_{\Omega} \|J_m f(\omega)\|_2 d \mu(\omega).
\end{equation*} By the reverse triangle inequality, we have that 
\begin{equation} \label{inequality-1}
    0 \leq \nn{f} - \nn{f}_m \leq \left( \sum_{|k| > m} \sqrt{a_k} \right) \|f\|_1 
\end{equation}
for every $f \in L_1(\mu)$. In particular, $\nn{\cdot}_m$ converges uniformly to $\nn{\cdot}$ on bounded subsets of $X$. We need the following key splitting lemma, which allows us to decompose any unit sphere element into an almost convex combination of elements with small supports. 

\begin{lemma} \label{lemma1:dual-has-the-DPr} For every $h \in S_X$ and every $\rho, \eta > 0$, there exist $h_1, \ldots, h_q \in L_1(\mu)$ such that
\begin{equation*} 
h = \sum_{i=1}^q h_i, \ \ \ \mu(\supp h_i) < \rho \ \ \ \mbox{and} \ \ \ 
\sum_{i=1}^q \nn{h_i} < 1 + \eta.
\end{equation*}
\end{lemma}

\begin{proof} Choose $m \in \N$ such that the first inequality below holds true and then $N > m$ such that the second inequality holds
\begin{equation*}
\nn{h}-\nn{h}_m\leq \left( \sum_{|k| > m} \sqrt{a_k} \right)\|h\|_1 < \frac{\eta}{2} \ \ \ \mbox{and}  \ \ \ \frac{2 \|h\|_1}{N} \sum_{|k| \leq m} |k| \sqrt{a_k} < \frac{\eta}{2}.
\end{equation*}
 Since $\mu$ is atomless, there exists a finite measurable partition $(B_1,\dots,B_q)$ of $\Omega$ such that $\mu(B_i) < \rho/N$ for every $i = 1,\ldots, q$ (see, for instance, \cite[Theorem 1.12.9 and Corollary 1.12.10]{Bog07}). Now, let us define $p_i: \Omega \rightarrow [0,1]$ by 
 \begin{equation*}
     p_i(\omega) := \frac{1}{N} \sum_{j=0}^{N-1} \chi_{B_i}(T^j \omega) 
 \end{equation*}
 for every $\omega \in \Omega$ and $h_i := h p_i$ for $i=1,\ldots, q$. Notice that $p_i \geq 0$ for every $i = 1, \ldots, q$. Also, since the sets $B_1, \ldots, B_q$ form a partition of $\Omega,$ for every $\omega \in \Omega$, \begin{equation*}
     \sum_{i=1}^q \chi_{B_i}(\omega) = 1
 \end{equation*} and thus 
 \begin{equation*}
     \sum_{i=1}^q p_i(\omega) = \frac{1}{N} \sum_{j=0}^{N-1} \sum_{i=1}^q \chi_{B_i}(T^j \omega) = \frac{1}{N} \sum_{j=0}^{N-1} 1 = 1.
 \end{equation*}
 So, \begin{equation*}
     \sum_{i=1}^q h_i = h \sum_{i=1}^q p_i = h.
 \end{equation*} Suppose that $h_i(\omega) \not=0$. Since $h_i(\omega) = h(\omega)p_i(\omega)$, we must have $h(\omega) \not=0$ and $p_i(\omega) > 0$. The condition $p_i(\omega) > 0$ means that there exists $j \in \{0, \ldots, N-1\}$ such that $\chi_{B_i}(T^j \omega) = 1$, that is, $T^j \omega \in B_i$ which is equivalent to say that $\omega \in T^{-j}(B_i)$. So,
 \begin{equation*}
     \supp h_i \subset \supp h \cap \bigcup_{j=0}^{N-1} T^{-j} (B_i).
 \end{equation*}
Since $T$ is measure preserving, $\mu(T^{-j} (B_i)) = \mu(B_i)$, and it follows that 
\begin{equation*}
    \mu(\supp h_i) \leq \mu \left( \bigcup_{j=0}^{N-1} T^{-j} (B_i) \right) \leq \sum_{j=0}^{N-1} \mu(T^{-j} (B_i)) < N \cdot \frac{\rho}{N} = \rho
\end{equation*}
for every $i=1,\ldots, q$. Now, for every $i=1,\ldots, q$, notice that 
\begin{equation*}
    p_i(T^k \omega) = \frac{1}{N} \sum_{j=0}^{N-1} \chi_{B_i} (T^{j+k} \omega)
\end{equation*}
for every $k \in \Z$. Assume first that $0 \leq k \leq m < N$. Then, for every $i=1,\ldots, q$,
\begin{equation*}
    p_i(T^k \omega) - p_i(\omega) = \frac{1}{N} \left( \sum_{j=N}^{N+k-1} \chi_{B_i} (T^j \omega) - \sum_{j=0}^{k-1} \chi_{B_i} (T^j \omega) \right)
\end{equation*}
which implies that 
\begin{equation*}
    \sum_{i=1}^q |p_i(T^k \omega) - p_i(\omega)| \leq \frac{1}{N} \sum_{i=1}^q \left( \sum_{j=N}^{N+k-1} \chi_{B_i}(T^j \omega) + \sum_{j=0}^{k-1} \chi_{B_i} (T^j \omega) \right) 
\end{equation*}
and since $B_1, \ldots, B_q$ is a partition, $\sum_{i=1}^q \chi_{B_i}(T^j \omega) = 1$ for every $j$. So, 
\begin{equation*}
    \sum_{i=1}^q |p_i(T^k \omega) - p_i(\omega)| \leq \frac{2k}{N}
\end{equation*}
for every $0 \leq k \leq m$. If $k < 0$, the same argument applies. So,
\begin{equation*}
    \sum_{i=1}^q |p_i(T^k \omega) - p_i(\omega)| \leq \frac{2|k|}{N}
\end{equation*}
for every $\omega \in \Omega$ and for every $|k| \leq m$. For every $|k| \leq m$, $h_i(T^k \omega) = h(T^k \omega) p_i(T^k \omega)$ and then $J_m h_i(\omega) = (\sqrt{a_k} h(T^k \omega) p_i(T^k \omega))_{|k| \leq m}$. We add and subtract $p_i(\omega)$ and write 
\begin{equation*}
    J_m h_i(\omega) = p_i(\omega) J_m h(\omega) + (\sqrt{a_k} h(T^k \omega) (p_i(T^k \omega) - p_i(\omega)))_{|k| \leq m}.
\end{equation*}
By the triangle inequality and the inequality we have obtained before, we get that 
\begin{equation*}
    \sum_{i=1}^q \nn{h_i}_m \leq \nn{h}_m + \frac{2 \|h\|_1}{N} \sum_{|k| \leq m} |k| \sqrt{a_k}
\end{equation*}
(notice that we have used that $\| \cdot \|_2 \leq \|\cdot\|_1$ above). Now, using (\ref{inequality-1}), we get 
\begin{equation*}
    \sum_{i=1}^q \nn{h_i} \leq \sum_{i=1}^q \nn{h_i}_m + \left( \sum_{|k| > m} \sqrt{a_k} \right) \sum_{i=1}^q \|h_i\|_1.
\end{equation*}
Since $\sum_{i=1}^q \|h_i\|_1 = \|h\|_1$ and $\nn{h}_m\leq \nn{h} = 1$, we obtain
\begin{equation*}
    \sum_{i=1}^q \nn{h_i} \leq 1 + \frac{2\|h\|_1}{N} \sum_{|k| \leq m} |k| \sqrt{a_k} + \left( \sum_{|k| > m} \sqrt{a_k} \right) \|h\|_1 
    < 1 + \frac{\eta}{2} + \frac{\eta}{2} = 1 + \eta.
\end{equation*}
\end{proof}

As a corollary, we immediately get the following.

\begin{lemma} \label{lemma2:dual-has-the-DPr-first-step-induction} Let $\Phi \in S_{X^*}$, $0 < \alpha < 1$ and $\rho > 0$ be given. Then, there exists $y \in S_X$ such that 
\begin{equation*}
    \re \Phi(y) > 1 - \alpha \ \ \ \mbox{and} \ \ \ \mu(\supp y) < \rho.
\end{equation*}
\end{lemma}

\begin{proof} Choose $h \in S_X$ such that $\re \Phi(h) > 1 - \frac{\alpha}{2}$. Applying Lemma \ref{lemma1:dual-has-the-DPr}, there are $h_1, \ldots, h_q \in X$ such that $h = \sum_{i=1}^q h_i$, $\mu(\supp h_i) < \rho$ and

\begin{equation*}
    \sum_{i=1}^q \nn{h_i} < 1 + \frac{\alpha}{2(1-\alpha)}.
\end{equation*}
Notice that this implies that there exists $i \in \{1, \ldots, q\}$ such that $\re \Phi(h_i) > (1-\alpha) \nn{h_i}$. In particular, $h_i \not=0$, $y:= \frac{h_i}{\nn{h_i}} \in S_X$ is well-defined, and $\re \Phi(y) > 1 - \alpha$. Moreover, $\mu (\supp y) = \mu ( \supp h_i) < \rho$.
\end{proof}

With this lemma at hand, it only remains to prove the following. 

\begin{lemma}
    Let $(g_n)\subset S_X$ be such that $\mu(\supp(g_n))\to 0$. Then, for every $f\in L_1(\mu)$, we have \begin{equation*}
       \lim_{n} \nn{f+g_n}=1+\nn{f}.
    \end{equation*}
\end{lemma}

\begin{proof}
    Fix $f\in L_1(\mu)$. For every $n\in\N$, let $B_n:=\supp(g_n)$ and $C_n:= \Omega\setminus B_n$. Fix $\eps>0$, and take $m\in\N$ such that for every $h\in \max\{1,\nn{f}\} B_X$, we have $\nn{h}-\nn{h}_m<\eps$. For every $n\in\N$, write \begin{equation*}
        \Omega_1^{(n)}:=\bigcap_{k=-m}^m\{\omega\in\Omega\colon T^k\omega\in C_n\},
    \end{equation*} and \begin{equation*}
        \Omega_2^{(n)}:=\Omega\setminus \Omega_1^{(n)}=\bigcup_{k=-m}^m\{\omega\in\Omega\colon T^k\omega\in B_n\}.
    \end{equation*} Note that since $T$ is measure preserving, we have \begin{equation*}
        \mu(\Omega_2^{(n)})\leq (2m+1)\mu(B_n)\to 0
    \end{equation*} as $n$ goes to infinity. Since $J_mg_n=0$ on $\Omega_1^{(n)}$, we have \begin{equation*}
        \nn{f+g_n}_m\geq \int_{\Omega_1^{(n)}}\norm{J_m(f)}_2d\mu(\omega)+\int_{\Omega_2^{(n)}}\norm{J_m(g_n)}_2d\mu(\omega)-\int_{\Omega_2^{(n)}}\norm{J_m(f)}_2d\mu(\omega).
    \end{equation*} Observe that since $\mu(\Omega_1^{(n)})\to 1$, we have that the first integral above converges to $\nn{f}_m$ as $n$ goes to $\infty$, while the third converges to $0$. Furthermore, the second integral is actually equal to $\nn{g_n}_m$ because $J_m(g_n)=0$ on $\Omega_1^{(n)}$. So there exists $N\in\N$ such that for every $n\geq N$,  \begin{equation*}
        \nn{f+g_n}_m\geq \nn{f}_m+\nn{g_n}_m-\eps. 
    \end{equation*} Thus, for every $n\geq N$, we have \begin{equation*}
        \nn{f+g_n}\geq \nn{f}+\nn{g_n}-4\eps=1+\nn{f}-4\eps. 
    \end{equation*} The conclusion follows.
\end{proof}

Thanks to the slice characterization of the DPr, Theorem~\ref{theorem:space-has-the-DPr} immediately follows from the above two lemmas.

\subsection{The dual satisfies the Daugavet property} Our goal in this subsection is to prove that in fact the space $X^*$ also has the DPr when $\mu$ is atomless. Again, notice that we do not need the extra assumption that the transformation $T: \Omega \rightarrow \Omega$ is ergodic.

\begin{theorem} \label{theorem:dual-has-the-DPr} Let $(\Omega, \Sigma, \mu)$ be an atomless probability space and $T: \Omega \rightarrow \Omega$ be an invertible measure-preserving transformation. Then the space $X^* = (L_{\infty}(\mu), \nn{\cdot}_{*})$ has the DPr.
\end{theorem}

First, we will need the following strengthening of Lemma~\ref{lemma2:dual-has-the-DPr-first-step-induction} from the previous section. 

\begin{lemma} \label{lemma2:dual-has-the-DPr} Let $\Phi \in S_{X^*}$, let $0 < \eta < 1$ and let $N \in \N$. There are $f_1, \ldots, f_N \in S_X$ whose supports are pairwise disjoint and such that 
\begin{equation*}
    \re \Phi(f_i) > 1 - \eta 
\end{equation*}
for every $i=1,\ldots, N$. 
\end{lemma}

\begin{proof} Choose $h \in S_X$ with $\re \Phi(h) > 1 - \frac{\eta}{4}$. Choose $\rho > 0$ such that $\mu(E) < \rho$ implies $\nn{h \chi_E} < \frac{\eta}{4}$ for $E \in \Sigma$. Suppose that $f_1, \ldots, f_{i-1}$ have been chosen for some $i\geq 2$, their supports $A_1, \ldots, A_{i-1}$ are pairwise disjoint and $\mu(A_j) < \rho/N$ (for $i=2$, this can be done thanks to Lemma~\ref{lemma2:dual-has-the-DPr-first-step-induction}). Put $E:= A_1 \cup \ldots \cup A_{i-1}$. Since $\mu(E) < \rho$, we have that 
\begin{equation*}
    u:= \frac{h \chi_{\Omega \setminus E}}{\nn{h \chi_{\Omega \setminus E}}} \in S_X \ \ \mbox{satisfies} \ \ \     \re \Phi(u) > 1 - \frac{\eta}{2}.
\end{equation*}
 Apply Lemma \ref{lemma1:dual-has-the-DPr} to $u$ and we obtain $u_1, \ldots, u_q \in X$ such that $u = \sum_{j=1}^q u_j$, where 
 \begin{equation*}
     \mu(\supp u_j) < \frac{\rho}{N} \ \ \ \mbox{and} \ \ \ \sum_{j=1}^q \nn{u_j} < 1 + \frac{\eta}{4}.
 \end{equation*}
Also, by the proof of Lemma \ref{lemma1:dual-has-the-DPr}, every $u_j$ is supported in $\supp u$. We want to show that there exists $j_0 \in \{1, \ldots, q\}$ such that $\re \Phi(\frac{u_{j_0}}{\nn{u_{j_0}}}) > 1 - \eta$, that is, $\re \Phi (u_{j_0}) > (1 - \eta)\nn{u_{j_0}}$ with $u_{j_0}\not=0$. Suppose that this is not the case. Then, $\re \Phi(u_j) \leq (1 - \eta) \nn{u_j}$ for every $j=1,\ldots, q$. So, 
\begin{equation*}
    \re \Phi(u) = \sum_{j=1}^q \re \Phi(u_j) \leq (1 - \eta) \sum_{j=1}^q \nn{u_j} < (1 - \eta)\left( 1 + \frac{\eta}{4} \right) < 1 - \frac{\eta}{2}
\end{equation*}
which yields a contradiction. So, we put $f_i := \frac{u_{j_0}}{\nn{u_{j_0}}}$ and induction gives us the desired result.
\end{proof}

We also need two estimates in $L_{\infty}(\mu)$. The first one is almost trivial, but we include the proof for the sake of completeness. 

\begin{lemma} \label{lemma3:dual-has-the-DPr} For $\Phi, \Psi \in L_{\infty}(\mu)$, we have that $\nn{\Phi \Psi}_{*} \leq \nn{\Phi}_{*} \|\Psi\|_{\infty}$.
\end{lemma}

\begin{proof} Pointwise, we have that $|\Psi(\omega) f(\omega)| \leq \|\Psi\|_{\infty} |f(\omega)|$ for every $\omega \in \Omega$. By Fact \ref{erg-monotone}, $\nn{ \Psi f} \leq \nn{ \|\Psi\|_{\infty} f} = \|\Psi\|_{\infty} \nn{f}$, for every $f \in X$. On the other hand, by the definition of the dual norm, $|\Phi(g)| \leq \nn{\Phi}_{*} \nn{g}$ for every $g \in X$. So,
\begin{equation*}
    |(\Phi \Psi)(f)| \leq \nn{\Phi}_{*} \nn{\Psi f} \leq \nn{\Phi}_{*} \|\Psi\|_{\infty} \nn{f}
\end{equation*}
for every $f \in L_1(\mu)$. Taking the supremum over all $f \in L_1(\mu)$ with $\nn{f} \leq 1$ gives the result.
\end{proof}

We need one more notation. For $A \in \Sigma$ and $n \in \N$, let us write 
\begin{equation*}
    A^{[n]} := \bigcup_{|r| \leq 2n} T^r A.
\end{equation*}

We have the following inequality.

\begin{lemma} \label{lemma4:dual-has-the-DPr} For $A \in \Sigma$, $n \in \N$ and $\Phi, \Psi \in X^*$, we have that 
\begin{equation*}
    \nn{\Phi \chi_A + \Psi \chi_{\Omega \setminus A^{[n]}}}_{*} \leq \left(1 + \sum_{|k| > n} \sqrt{a_k} \right) \max \{ \nn{\Phi}_{*}, \nn{\Psi}_*\}.
\end{equation*}
\end{lemma}

\begin{proof} Fix $f \in X$. For a fixed $\omega \in \Omega$, the vector $J_n( f \chi_A )(\omega)$ is nonzero if there exists $|k| \leq n$ such that $f( T^k \omega ) \chi_A( T^k \omega ) \not= 0$. In particular, $T^k \omega \in A$. Similarly, the vector $J_n( f \chi_{\Omega \setminus A^{[n]}} )(\omega)$ is nonzero if there exists $|l| \leq n$ such that $T^l \omega \in \Omega \setminus A^{[n]}$, that is, $T^l \omega \notin A^{[n]}$. For a fixed $\omega$, the vectors $J_n( f \chi_A )(\omega)$ and $J_n( f \chi_{\Omega \setminus A^{[n]}} )(\omega)$ cannot be both nonzero. Indeed, in this case, $T^k \omega \in A$ and $T^l \omega \notin A^{[n]}$ for some $|k|, |l| \leq n$. However, $T^l \omega = T^{l-k}( T^k \omega)$ and $|l-k| \leq 2n$. Since $T^k \omega \in A$, it follows that $T^{l-k}( T^k \omega ) \in T^{l-k} A$. Also,
\begin{equation*}
    T^{l-k} A \subseteq \bigcup_{|r| \leq 2n} T^r A = A^{[n]}.
\end{equation*}
This would imply that $T^l \omega \in A^{[n]}$, a contradiction. Now, notice that each vector above is obtained from $J_n f(\omega)$ by replacing some coordinates with zero, so $\| J_n( f \chi_A )(\omega) \|_2\leq \| J_n f(\omega) \|_2$ and $\| J_n( f \chi_{\Omega \setminus A^{[n]}} )(\omega) \|_2\leq \| J_n f(\omega) \|_2$, and by the above observation it follows that also
\begin{equation*}
    \| J_n( f \chi_A )(\omega) \|_2 + \| J_n( f \chi_{\Omega \setminus A^{[n]}} )(\omega) \|_2
    \leq \| J_n f(\omega) \|_2.
\end{equation*}
This implies that
\begin{equation*}
    \nn{ f \chi_A }_n + \nn{f \chi_{\Omega \setminus A^{[n]}}}_n \leq \nn{f}.
\end{equation*}
So, using (\ref{inequality-1}), we get 
\begin{eqnarray*}
    \nn{f \chi_A} + \nn{f \chi_{\Omega \setminus A^{[n]}}} &\leq& \nn{f} + \left(        \sum_{|k| > n} \sqrt{a_k}\right) \left( \| f \chi_A \|_1 + \| f \chi_{\Omega \setminus A^{[n]}} \|_1 \right).
\end{eqnarray*}
Since $A \subseteq A^{[n]}$, the sets $A$ and $\Omega \setminus A^{[n]}$ are disjoint. Hence,
\begin{equation*}
    \| f \chi_A \|_1  + \| f \chi_{\Omega \setminus A^{[n]}} \|_1 \leq \|f\|_1 \leq \nn{f}.
\end{equation*}
Therefore,
\begin{equation*}
    \nn{f \chi_A} + \nn{f \chi_{\Omega \setminus A^{[n]}}} \leq \left( 1 + \sum_{|k| > n} \sqrt{a_k} \right) \nn{f}.
\end{equation*}

Now, since
\begin{equation*}
    \left(\Phi \chi_A + \Psi \chi_{\Omega \setminus A^{[n]}}\right)(f) = \Phi( f \chi_A )
    + \Psi( f \chi_{\Omega \setminus A^{[n]}} ),
\end{equation*}
we have that
\begin{eqnarray*}
    \left| \left(\Phi \chi_A +\Psi \chi_{\Omega \setminus A^{[n]}}  \right)(f) \right|
    &\leq& \nn{\Phi}_{*} \nn{f \chi_A} + \nn{\Psi}_{*} \nn{f \chi_{\Omega \setminus A^{[n]}}}
    \\
    &\leq&  \max \{ \nn{\Phi}_{*}, \nn{\Psi}_{*} \} \left( \nn{f \chi_A}  +  \nn{f \chi_{\Omega \setminus A^{[n]}}} \right) \\
    &\leq& \left( 1 + \sum_{|k| > n} \sqrt{a_k} \right) \max \{ \nn{\Phi}_{*}, \nn{\Psi}_{*} \} \nn{f}.
\end{eqnarray*}
Taking the supremum over all $f \in X$ with $\nn{f} \leq 1$ gives the result.
\end{proof}

Now we are ready to prove Theorem \ref{theorem:dual-has-the-DPr}.

\begin{proof}[Proof of Theorem \ref{theorem:dual-has-the-DPr}] Fix $\Phi \in S_{X^*}$, $\eps \in (0,1)$ and a slice $S := S( B_{X^*}, F, \alpha )$ with $F \in S_{X^{**}}$ and $0 < \alpha < 1$. Choose $\Psi_0 \in S_{X^*}$ and then $n, N \in \N$ so that 
\begin{equation} \label{eq:dual-DPr-choice}
    \re F(\Psi_0) > 1 - \frac{\alpha}{4}, \ \ \ \sum_{|k| > n} \sqrt{a_k} < \min \left\{
        \frac{\alpha}{4},\frac{\eps}{4} \right\} \ \ \ \mbox{and} \ \ \ \frac{4n+2}{N} < \frac{\alpha}{4}.
\end{equation}
By Lemma \ref{lemma2:dual-has-the-DPr}, there are $f_1, \ldots, f_N \in S_X$ with pairwise disjoint supports $A_i := \supp f_i$ such that 
\begin{equation} \label{eq:dual-DPr-fi}
    \re \Phi(f_i) > 1 - \frac{\eps}{4}
    \end{equation}
for every $i = 1, \ldots, N$. Now, let us set
\begin{equation*}
    A_i^{[n]} := \bigcup_{|r| \leq 2n} T^r A_i \ \ \ \mbox{and} \ \ \ \Psi_i
    := \Phi \chi_{A_i} - \Psi_0 \chi_{A_i^{[n]}}
\end{equation*}
for every $i = 1, \ldots, N$. For a fixed $r$, the sets $T^r A_1, \ldots, T^r A_N$ are pairwise disjoint. Then,
\begin{equation*}
    \sum_{i=1}^N \chi_{A_i^{[n]}} \leq \sum_{|r| \leq 2n} \sum_{i=1}^N \chi_{T^r A_i} \leq
    4n+1.
\end{equation*}
Pick scalars $\lambda_i$ of modulus one so that $\lambda_i F(\Psi_i) = |F(\Psi_i)|$ for every $i = 1, \ldots, N$. If $F(\Psi_i) = 0$, take $\lambda_i = 1$. Since the sets $A_1, \ldots, A_N$ are pairwise disjoint,
\begin{equation*}
    \left\| \sum_{i=1}^N \lambda_i \chi_{A_i} \right\|_{\infty} \leq  1
\end{equation*}
and the preceding estimate gives
\begin{equation*}
    \left\| \sum_{i=1}^N \lambda_i \chi_{A_i^{[n]}} \right\|_{\infty} \leq 4n+1.
\end{equation*}
Lemma \ref{lemma3:dual-has-the-DPr} now gives
\begin{eqnarray*}
    \sum_{i=1}^N |F(\Psi_i)|  &=& \left|  F\left(\sum_{i=1}^N \lambda_i \Psi_i \right)
    \right|  \\
    &\leq& \nn{ \sum_{i=1}^N \lambda_i \Psi_i }_{*}  \\
    &\leq& \nn{ \Phi \left( \sum_{i=1}^N \lambda_i \chi_{A_i} \right)}_{*}  + \nn{
        \Psi_0\left( \sum_{i=1}^N \lambda_i \chi_{A_i^{[n]}} \right)}_{*}  
    \leq  1 + (4n+1)  =  4n+2.
\end{eqnarray*}
Since $\frac{4n+2}{N} < \frac{\alpha}{4}$ by (\ref{eq:dual-DPr-choice}), there exists
$i \in \{1, \ldots, N\}$ such that 
\begin{equation} \label{eq:dual-DPr-F-Psi-i}
    |F(\Psi_i)| < \frac{\alpha}{4}.
\end{equation}
For this $i$, define
\begin{equation*}
    \widetilde{\Psi} :=   \Psi_0 + \Psi_i.
    \end{equation*}
Since $A_i \subseteq A_i^{[n]}$, we have
\begin{equation*}
    \widetilde{\Psi} = \Psi_0  +\Phi \chi_{A_i} - \Psi_0 \chi_{A_i^{[n]}}  = \Phi \chi_{A_i}
    + \Psi_0 \chi_{\Omega \setminus A_i^{[n]}}.
\end{equation*}
Lemma \ref{lemma4:dual-has-the-DPr} now gives
\begin{equation*}
    \nn{\widetilde{\Psi}}_{*} \leq 1 + \sum_{|k| > n} \sqrt{a_k}.
\end{equation*}
Now, define
\begin{equation*}
    \Psi := \frac{ \widetilde{\Psi}}{ 1 + \sum_{|k| > n} \sqrt{a_k}}.
\end{equation*}
It follows that $\Psi \in B_{X^*}$. Also,
\begin{equation*}
    \re F(\widetilde{\Psi}) = \re F(\Psi_0)  + \re F(\Psi_i) \geq \re F(\Psi_0) - |F(\Psi_i)| > 1 - \frac{\alpha}{4} - \frac{\alpha}{4} = 1 - \frac{\alpha}{2},
\end{equation*}
where we have used (\ref{eq:dual-DPr-choice}) and (\ref{eq:dual-DPr-F-Psi-i}). Consequently,
\begin{equation*}
    \re F(\Psi) > \frac{ 1 - \frac{\alpha}{2}}{1 + \sum_{|k| > n} \sqrt{a_k} }
    > \frac{ 1 - \frac{\alpha}{2} }{  1 + \frac{\alpha}{4}  } > 1 - \alpha,
\end{equation*}
where the second inequality follows from (\ref{eq:dual-DPr-choice}). Thus, $\Psi \in S$. Since $f_i$ is supported on $A_i$, we have $\widetilde{\Psi}(f_i) = \Phi(f_i)$, since $\Psi_0 \chi_{\Omega \setminus A_i^{[n]}}$ vanishes on $A_i$. Therefore,
\begin{eqnarray*}
    \nn{\Phi + \Psi}_{*}  &\geq& \re(\Phi + \Psi)(f_i) \\
    &=& \re \Phi(f_i) + \frac{ \re \widetilde{\Psi}(f_i) }{ 1 + \sum_{|k| > n} \sqrt{a_k}
    } \\
    &=& \left(  1 + \frac{ 1 }{1 + \sum_{|k| > n} \sqrt{a_k} } \right)
    \re \Phi(f_i) > \left(2 - \frac{\eps}{4}\right) \left(  1 - \frac{\eps}{4} \right)  > 2 - \eps,
\end{eqnarray*}
where the last two inequalities follow from (\ref{eq:dual-DPr-choice}) and (\ref{eq:dual-DPr-fi}). The slice characterization of the DPr gives the desired result.
\end{proof}

\subsection{Further remarks}

There is yet another way to see that $(L_1(\mu), \nn{\cdot})$ satisfies the DPr. Indeed, it turns out that this space actually satisfies the stronger operator Daugavet property (introduced in \cite{RTV21}). We say that $X$ has the {\bf operator Daugavet property} (ODPr, for short) if, for every $x_1,\dots,x_n\in S_X$, every slice $S$ of $B_X$, and every $\eps>0$, there exists $y\in S$ such that, for every $z\in B_X$, there exists an operator $T\in\mathcal{L}(X)$ satisfying
\begin{equation*}
    \|T\|\leq 1+\eps, \ \ \ \|T(x_i)-x_i\|<\eps \ \forall \ 1\leq i\leq n \ \ \ \mbox{and} \ \ \  T(y)=z.
\end{equation*}
It is known that the ODPr implies the DPr, but the converse is false in general, as witnesses the recent example from \cite{Ham26-2}.

\begin{theorem}\label{theorem:X_has_ODP} Let $(\Omega, \Sigma, \mu)$ be an atomless probability space. Let $T: \Omega \rightarrow \Omega$ be an invertible measure-preserving transformation. Then the space $X = (L_1(\mu), \nn{\cdot})$ has the ODPr. 
\end{theorem}

\begin{proof} Indeed, let us fix $x_1, \ldots, x_m \in S_X$, a slice $S = S( B_X, \Phi, \alpha )$ with $\Phi \in S_{X^*}$, $\eps \in (0,1)$ and $0 < \alpha < 1$. Choose $n \in \N$ such that
\begin{equation*}
    \sum_{|k| > n} \sqrt{a_k} < \eps.
\end{equation*}
Since $T$ is measure preserving,
\begin{equation*}
    \mu(A^{[n]}) \leq \sum_{|r| \leq 2n} \mu(T^r A) = (4n+1) \mu(A).
\end{equation*}
This means that $\mu(A^{[n]})$ can be made arbitrarily small by making $\mu(A)$ sufficiently small, while $\nn{x_i \chi_E} \rightarrow 0$ as $\mu(E) \rightarrow 0$. Therefore, by Lemma \ref{lemma2:dual-has-the-DPr-first-step-induction}, we may choose $y \in S \cap S_X$ with support $A := \supp y$ so small that $2 \nn{x_i \chi_{A^{[n]}}} < \eps$ for every $i = 1, \ldots, m$. Choose $\Psi \in S_{X^*}$ such that $\Psi(y) = 1$. Now, let
$z \in B_X$ be arbitrary. Define $R_z: X \rightarrow X$ by
\begin{equation*}
    R_z(f) := f \chi_{\Omega \setminus A^{[n]}} + \Psi( f \chi_A ) z
\end{equation*}
for every $f \in X$. By the estimate obtained in the proof of Lemma \ref{lemma4:dual-has-the-DPr}, we have that
\begin{eqnarray*}
    \nn{R_z(f)} &\leq& \nn{f \chi_{\Omega \setminus A^{[n]}}} + |\Psi( f \chi_A )| \nn{z}
    \\
    &\leq& \nn{f \chi_{\Omega \setminus A^{[n]}}} + \nn{f \chi_A} \\
    &\leq& \left( 1 + \sum_{|k| > n} \sqrt{a_k}\right) \nn{f}  < (1+\eps) \nn{f}
\end{eqnarray*}
for every $f \in X$. This implies that $\|R_z\| \leq 1+\eps$. Now, since $y$ is supported on $A$ and $A \subseteq A^{[n]}$, we have that $y \chi_{\Omega \setminus A^{[n]}} = 0$ and $y \chi_A =  y$. So, $R_z(y) = \Psi(y) z = z$. Finally, for $1 \leq i \leq m$,
\begin{equation*}
    R_z(x_i) - x_i = x_i \chi_{\Omega \setminus A^{[n]}} + \Psi( x_i \chi_A ) z - x_i  = - x_i \chi_{A^{[n]}} + \Psi( x_i \chi_A ) z.
\end{equation*}
Hence, using $A \subseteq A^{[n]}$, monotonicity of $\nn{\cdot}$, $\nn{\Psi}_{*} = 1$ and $\nn{z} \leq 1$, we get
\begin{eqnarray*}
    \nn{R_z(x_i) - x_i} &\leq& \nn{x_i \chi_{A^{[n]}}} + |\Psi( x_i \chi_A )| \nn{z} \\
    &\leq& \nn{x_i \chi_{A^{[n]}}} + \nn{x_i \chi_A} \\
    &\leq& 2 \nn{x_i \chi_{A^{[n]}}} <  \eps.
\end{eqnarray*}
\end{proof}

As far as we know, it is still an open problem to find a bidual Banach space with the Daugavet property. As we can see below, our space does not yield such an example.

\begin{fact}
The bidual of $X := (L_1(\mu), \nn{\cdot})$ fails the DPr.
\end{fact}

\begin{proof}
Notice that if $\Phi, \Psi \in L_{\infty}(\mu)$ and $|\Phi| \leq |\Psi|$, then, with $\Phi/\Psi = 0$ on $\{\Psi = 0\}$,
\begin{equation*}
    \nn{\Phi}_{*} = \sup_{\nn{f} \leq 1} \left| \int_{\Omega} f \Phi d\mu \right| 
    = \sup_{\nn{f} \leq 1} \left| \int_{\Omega} \frac{\Phi}{\Psi} f \Psi d\mu \right| \leq \nn{\Psi}_{*},
\end{equation*}
because by monotonicity of $\nn{\cdot}$, we have that $\nn{\frac{\Phi}{\Psi} f}\leq\nn{f}$ for every $f\in L_1(\mu)$.

By the Gelfand representation theorem, there are a compact Hausdorff space $K$ and a surjective $*$-isomorphism
$\Gamma: L_{\infty}(\mu) \rightarrow C(K)$. For $U \in C(K)$, write $\nn{U}_{*} := \nn{\Gamma^{-1}(U)}_{*}$. By the above, $|U| \leq |V|$ implies $\nn{U}_{*} \leq \nn{V}_{*}$. For $\nu \in M(K)$, set
\begin{equation*}
    \nn{\nu}_{**}  := \sup_{\nn{U}_{*} \leq 1} \left| \int_K U \, d\nu \right|.
\end{equation*}
By the Riesz representation theorem, the mapping $\Gamma^*: (M(K), \nn{\cdot}_{**}) \rightarrow
    X^{**}$ defined by
\begin{equation*}
   (\Gamma^* \nu)(\Phi) = \int_K \Gamma(\Phi)  d\nu
   \end{equation*}
is a surjective linear isometry. Suppose that $\lambda, \nu \in M(K)$ and $|\lambda| \leq |\nu|$. If $\nn{U}_{*} \leq 1$, then
\begin{equation*}
    \left|   \int_K U d\lambda \right| \leq  \int_K |U| d|\nu| =  \sup_{\substack{V \in C(K)\\ |V| \leq |U|}} \left|  \int_K V d\nu \right| \leq \nn{\nu}_{**}.
\end{equation*}
Therefore, $|\lambda| \leq |\nu|$ implies $\nn{\lambda}_{**} \leq \nn{\nu}_{**}$. 

Finally, fix $t \in K$ and define $P: M(K) \rightarrow M(K)$ by $P\nu := \nu(\{t\}) \delta_t$ for every $\nu \in M(K)$. Then $P$ is a nonzero rank-one projection and the total variations satisfy $|P\nu| = |\nu|\!\upharpoonright_{\{t\}} \leq |\nu|$ and $|(I-P)\nu| = |\nu|\!\upharpoonright_{K\setminus\{t\}} \leq |\nu|$. It follows that $\nn{P} = 1$ and $\nn{I-P} \leq 1$, where $\nn{\cdot}$ is the operator norm with respect to $\nn{\cdot}_{**}$ . Therefore, $(M(K),\nn{\cdot}_{**})$, and hence $X^{**}$, fails the DPr.
\end{proof}

It follows from \cite[Proposition~2.5]{KMMP09} that the norm of a Banach space with the DPr cannot be Fr\'echet smooth, but we do not know if the norm $\nn{\cdot}$ satisfies any property of smoothness in between Gâteaux and Fr\'echet. Concerning rotundity properties, let us conclude the paper with the following remark. 

\begin{remark} It is easy to prove that the norm of a Banach space with the DPr cannot be locally uniformly rotund (LUR) and in fact cannot even be weakly locally uniformly rotund (wLUR). However, to the best of our knowledge, it is an open problem whether such a norm can be midpoint locally uniformly rotund. Recall that a Banach space $X$ is said to be midpoint locally uniformly rotund (MLUR, for short) if, whenever $x \in S_X$ and $(h_n) \subseteq X$ satisfy $\|x+h_n\| \longrightarrow 1$ and $\|x-h_n\| \longrightarrow 1$, then $\|h_n\| \longrightarrow 0$. We observe that, when $T$ is ergodic, the norm $\nn{\cdot}$ is not MLUR. Indeed, by Rokhlin's lemma (see for instance \cite[Theorem 2.3]{Wei22}), for every $n \in \N$, there exists $B_n \in \Sigma$ such that $B_n, TB_n, \ldots, T^{2n-1}B_n$ are pairwise disjoint and
\begin{equation*}
    \mu \left( \bigcup_{j=0}^{2n-1} T^j B_n \right) > 1-\frac{1}{n}.
\end{equation*}
Let
\begin{equation*}
    A_n := \bigcup_{j=0}^{n-1} T^j B_n.
\end{equation*}
Since the sets $B_n, TB_n, \ldots, T^{2n-1}B_n$ are pairwise disjoint, we have that 
\begin{equation*}
    \mu(A_n) = n\mu(B_n) = \frac{1}{2} \mu \left( \bigcup_{j=0}^{2n-1} T^j B_n \right) \longrightarrow \frac{1}{2}.
\end{equation*}
Moreover, $T^{-1}A_n \Delta A_n \subseteq T^{-1}B_n \cup T^{n-1}B_n$. Since $2n\mu(B_n)\leq 1$, we have
\begin{equation*}
    \mu(T^{-1}A_n \Delta A_n) \leq 2\mu(B_n) \leq \frac{1}{n}.
\end{equation*}
By iteration and measure preservation, we obtain
\begin{equation*}
    \mu(T^{-k}A_n \Delta A_n) \leq \frac{|k|}{n}
\end{equation*}
for every $k \in \Z$. Hence, for every fixed $k \in \Z$, $\mu(T^{-k}A_n \Delta A_n) \longrightarrow 0$. Since $\sum_{k \in \Z}\sqrt{a_k}<\infty$ and
$\mu(T^{-k}A_n \Delta A_n)\leq 1$, the dominated convergence theorem gives
\begin{equation*}
    \sum_{k \in \Z} \sqrt{a_k} \mu(T^{-k}A_n \Delta A_n) \longrightarrow 0.
\end{equation*}
Since $\sum_{k \in \Z}a_k=1$, for almost every $\omega \in \Omega$ we have
\begin{equation*}
    \left| \|J\chi_{A_n}(\omega)\|_2-\chi_{A_n}(\omega) \right| \leq \left(
        \sum_{k \in \Z}  a_k \left| \chi_{A_n}(T^k\omega)-\chi_{A_n}(\omega) \right|^2 \right)^{1/2} \leq \sum_{k \in \Z} \sqrt{a_k} \left| \chi_{A_n}(T^k\omega)-\chi_{A_n}(\omega) \right|.
\end{equation*}
Integrating, we obtain
\begin{equation*}
    \left| \nn{\chi_{A_n}}-\mu(A_n) \right| \leq \sum_{k \in \Z} \sqrt{a_k}   \mu(T^{-k}A_n \Delta A_n) \longrightarrow   0.
\end{equation*}
Therefore, $\nn{\chi_{A_n}} \longrightarrow \frac{1}{2}$. Since $T^{-k}(\Omega \setminus A_n) \Delta (\Omega \setminus A_n) = T^{-k}A_n \Delta A_n$, the same argument gives $\nn{\chi_{\Omega \setminus A_n}} \longrightarrow \frac{1}{2}$. Finally, let $h_n := \chi_{A_n}-\chi_{\Omega \setminus A_n}$. Since $|h_n|=\chi_\Omega$, we have $\nn{h_n}=1$ for every $n \in \N$. On the other hand, $\nn{\chi_\Omega+h_n} = 2\nn{\chi_{A_n}} \longrightarrow 1$ and  $\nn{\chi_\Omega-h_n} = 2\nn{\chi_{\Omega \setminus A_n}} \longrightarrow 1$. Since $\nn{\chi_\Omega}=1$ while $\nn{h_n}=1$ for every $n \in \N$, $\chi_\Omega$ is not an MLUR point.
\end{remark}

 \section*{Acknowledgments} The authors are deeply grateful to Petr Hájek for suggesting several ideas concerning possible constructions of smooth or strictly convex renormings with the Daugavet property. They are also thankful to Johann Langemets and Märt Põldvere for many helpful conversations on the topic of this paper. Part of this work was carried out during a research stay of Yoël Perreau in Prague, Czech Republic. He is grateful to the Department of Mathematics of the Faculty of Electrical Engineering at the Czech Technical University in Prague for its hospitality. He also thanks Michal Doucha for all his help during this period.

\section*{AI disclosure statement} The authors used OpenAI's ChatGPT Plus during the exploratory stage of this work. After several days of active interaction concerning the problem of constructing a strictly convex renorming of $L_1[0,1]$, inspired by the recent renorming construction of Cobollo and Hájek \cite{CH25}, ChatGPT suggested considering the norm introduced in this paper. The authors did not rely on the subsequent arguments or proofs produced by ChatGPT, which appeared to involve tools from dynamical systems, an area outside the authors' expertise. Instead, taking only the suggested norm as a starting point, the authors carried out the mathematical analysis independently and obtained stronger results than what it claimed, including that the dual has the DPr, that the space satisfies the ODPr, and that ergodicity of $T$, wMLUR, strict convexity and Gâteaux smoothness are all equivalent. Moreover, the direct proof for the Daugavet property in the space is also given by the authors.

\section*{Funding}
 
S. Dantas has been supported by the grants PID2021-122126NB-C31 and PID2021-122126NB-C33 funded by MICIU/AEI/ 10.13039/ 501100011033 and by ERDF/EU. J.K. Kaasik was supported by the Estonian Research Council grant PRG1901. Y. Perreau has been supported by the Estonian Research Council grant PRG2545.


\begin{thebibliography}{FaHaMo}

\bibitem[AALMPPV24]{AALMPPV24}
\textsc{T.~A.~Abrahamsen}, \textsc{R.~J.~Aliaga}, \textsc{V.~Lima}, \textsc{A.~Martiny}, \textsc{Y.~Perreau},
\textsc{A.~Proch\'azka}, and \textsc{T.~Veeorg},
\emph{Delta-points and their implications for the geometry of Banach spaces},
J. Lond. Math. Soc. (2) \textbf{109} (2024), no.~5, Paper No.~e12913, 38 pp.

\bibitem[AHLT26]{AHLT26}
\textsc{T.~A.~Abrahamsen}, \textsc{P.~H\'ajek}, \textsc{V.~Lima}, and \textsc{S.~Troyanski},
\emph{Strictly convex norms and the local diameter two property},
Rev. Mat. Complut. \textbf{39} (2026), 857--877.

\bibitem[AHNTT16]{AHNTT16}
\textsc{T.~A.~Abrahamsen}, \textsc{P.~H\'ajek}, \textsc{O.~Nygaard}, \textsc{J.~Talponen}, and \textsc{S.~Troyanski},
\emph{Diameter 2 properties and convexity},
Studia Math. \textbf{232} (2016), no.~3, 227--242.

\bibitem[AHLP20]{AHLP20}
\textsc{T.~A.~Abrahamsen}, \textsc{R.~Haller}, \textsc{V.~Lima}, and \textsc{K.~Pirk},
\emph{Delta- and Daugavet-points in Banach spaces},
Proc. Edinb. Math. Soc. (2) \textbf{63} (2020), no.~2, 475--496.

\bibitem[ALMP22]{ALMP22}
\textsc{T.~A.~Abrahamsen}, \textsc{V.~Lima}, \textsc{A.~Martiny}, and \textsc{Y.~Perreau},
\emph{Asymptotic geometry and Delta-points},
Banach J. Math. Anal. \textbf{16} (2022), no.~4, Paper No.~57.

\bibitem[BGM05]{BGM05}
\textsc{J.~Becerra Guerrero} and \textsc{M.~Mart\'in},
\emph{The Daugavet property of $C^*$-algebras, $JB^*$-triples, and of their isometric preduals},
J. Funct. Anal. \textbf{224} (2005), no.~2, 316--337.

\bibitem[Bog07]{Bog07}
\textsc{V.~I.~Bogachev},
\emph{Measure Theory, Vol.~I},
Springer, Berlin, 2007.

\bibitem[CH25]{CH25}
\textsc{Ch.~Cobollo} and \textsc{P.~H\'ajek},
\emph{Octahedrality and G\^ateaux smoothness},
J. Math. Anal. Appl. \textbf{543} (2025), Art.~128968.

\bibitem[Dau63]{Daugavet63}
\textsc{I.~K.~Daugavet},
\emph{A property of complete continuous operators in the space {{\(C\)}}},
Usp. Mat. Nauk \textbf{18} (1963), no.~5(113), 157--158 (Russian).

\bibitem[DBPS25]{DBPS25}
\textsc{C.~A.~De~Bernardi}, \textsc{A.~Preti}, and \textsc{J.~Somaglia},
\emph{A note on smooth rotund norms which are not midpoint locally uniformly rotund},
J. Math. Anal. Appl. \textbf{550} (2025), no.~2, Art.~129544.

\bibitem[DGZ93]{DGZ93}
\textsc{R. Deville}, \textsc{G. Godefroy} and \textsc{V. Zizler}, \emph{Smoothness and renormings in Banach spaces}, Harlow: Longman Scientific \& Technical; New York: John Wiley \& Sons, Inc. (1993)

\bibitem[DS58]{DunfordSchwartz}
\textsc{N.~Dunford} and \textsc{J.~T.~Schwartz},
\emph{Linear Operators. Part I: General Theory},
Pure and Applied Mathematics, Vol.~7,
Interscience Publishers, New York, 1958.
 
\bibitem[GPRZ18]{GPRZ18}
\textsc{L.~Garc\'ia-Lirola}, \textsc{A.~Proch\'azka}, and \textsc{A.~Rueda Zoca},
\emph{A characterisation of the Daugavet property in spaces of Lipschitz functions},
J. Math. Anal. Appl. \textbf{464} (2018), no.~1, 473--492.

\bibitem[HLPV23]{HLPV23}
  R.~Haller, J.~Langemets, Y.~Perreau, and T.~Veeorg,
  \emph{Unconditional bases and {D}augavet renormings},
  J. Funct. Anal. 286 (2024), no.~12, Paper No. 110421,
  31. 
  
\bibitem[Ham26-1]{Ham26-1}
\textsc{S.~Hamad},
\emph{About smooth and non-poor subspaces of Daugavet spaces},
preprint (2026), arXiv:2604.24529.

\bibitem[Ham26-2]{Ham26-2}
\textsc{S.~Hamad},
\emph{On the weak operator Daugavet property and poor subspaces},
preprint (2026), arXiv:2608.11098.

\bibitem[Kad96]{Kad96}
\textsc{V.~M.~Kadets},
\emph{Some remarks concerning the Daugavet equation},
Quaest. Math. \textbf{19} (1996), no.~1--2, 225--235.

\bibitem[KMMP09]{KMMP09}
\textsc{V.~Kadets}, \textsc{M.~Mart\'in}, \textsc{J.~Mer\'i}, and \textsc{R.~Pay\'a},
\emph{Convexity and smoothness of Banach spaces with numerical index one},
Illinois J. Math. \textbf{53} (2009), no.~1, 163--182.

\bibitem[KMRZW25]{KMRZW25}
\textsc{V.~Kadets}, \textsc{M.~Mart\'{\i}n}, \textsc{A.~Rueda~Zoca}, and \textsc{D.~Werner},
\emph{Banach spaces with the {Daugavet} property},
preliminary version (2025). Available at DIGIBUG
\url{https://hdl.handle.net/10481/104200}.

\bibitem[KSSW00]{KSSW00}
\textsc{V.~M.~Kadets}, \textsc{R.~V.~Shvidkoy}, \textsc{G.~G.~Sirotkin}, and \textsc{D.~Werner},
\emph{Banach spaces with the {Daugavet} property},
Trans. Amer. Math. Soc. \textbf{352} (2000), no.~2, 855--873.

\bibitem[KW04]{KW04}
\textsc{V.~M.~Kadets} and \textsc{D.~Werner},
\emph{A Banach space with the Schur and the Daugavet property},
Proc. Amer. Math. Soc. \textbf{132} (2004), no.~6, 1765--1773.

\bibitem[LRZ21]{LRZ2021}
\textsc{G.~L\'opez-P\'erez} and \textsc{A.~Rueda Zoca},
\emph{$L$-orthogonality, octahedrality and Daugavet property in Banach spaces},
Adv. Math. \textbf{383} (2021), Paper No.~107719.

\bibitem[MPRZ24]{MPRZ24}
M.~Mart\'in, Y.~Perreau, and A.~Rueda~Zoca, \emph{Diametral notions for elements of the unit ball of a {B}anach space}, Dissertationes Math.
\textbf{594} (2024), 61pp.

\bibitem[MRZ22]{MRZ22}
\textsc{M.~Mart\'in} and \textsc{A.~Rueda Zoca},
\emph{Daugavet property in projective symmetric tensor products of Banach spaces},
Banach J. Math. Anal. \textbf{16} (2022), Art.~35.

\bibitem[MU14]{MU14}
\textsc{M.~Mart\'in} and \textsc{Y.~Ueda},
\emph{On the geometry of von Neumann algebra preduals},
Positivity \textbf{18} (2014), no.~3, 519--530.

\bibitem[MRZ25]{MRZ25}
\textsc{R.~Medina} and \textsc{A.~Rueda Zoca},
\emph{A characterisation of the Daugavet property in spaces of vector-valued Lipschitz functions},
J. Funct. Anal. \textbf{289} (2025), no.~1, Art.~110896.

\bibitem[NPTV24]{NPTV24}
\textsc{O.~Nygaard}, \textsc{M.~P\~oldvere}, \textsc{S.~Troyanski}, and \textsc{T.~Viil},
\emph{Strictly convex renormings and the diameter 2 property},
Studia Math. \textbf{274} (2024), no.~1, 37--49.

\bibitem[Oik02]{Oik02}
\textsc{T.~Oikhberg},
\emph{The Daugavet property of $C^*$-algebras and non-commutative $L_p$-spaces},
Positivity \textbf{6} (2002), no.~1, 59--73.
 

\bibitem[RTV21]{RTV21}
\textsc{A.~Rueda Zoca}, \textsc{P.~Tradacete}, and \textsc{I.~Villanueva},
\emph{Daugavet property in tensor product spaces},
J. Inst. Math. Jussieu \textbf{20} (2021), no.~4, 1409--1428.

\bibitem[Shv00]{Shv00}
\textsc{R.~V.~Shvidkoy},
\emph{Geometric aspects of the Daugavet property},
J. Funct. Anal. \textbf{176} (2000), no.~2, 198--212.

\bibitem[Tal90]{Tal90}
\textsc{M.~Talagrand},
\emph{The three-space problem for $L_1$},
J. Amer. Math. Soc. \textbf{3} (1990), no.~1, 9--29.

\bibitem[Wei22]{Wei22}
\textsc{C.~Wei\ss},
\emph{Systems of rank one, explicit Rokhlin towers, and covering numbers},
Arch. Math. \textbf{118} (2022), 181--193.

\end{thebibliography}
\end{document}